\documentclass[11pt,a4paper]{article}

\usepackage{mathtools}
\usepackage{amsthm,amsmath}
\usepackage{amssymb,MnSymbol}
\usepackage[svgnames]{xcolor}
\usepackage{tabto}
\usepackage{paralist}
\usepackage[
colorlinks=true,
linkcolor=MidnightBlue,
citecolor=MidnightBlue,
urlcolor=MidnightBlue,
pagebackref=true]{hyperref}
\usepackage[alphabetic]{amsrefs}
\usepackage{graphicx}
\usepackage{pinlabel}
\usepackage{enumitem}
\usepackage[norefs]{refcheck}
\usepackage[margin=3.3cm]{geometry}

\newcommand{\N}{\mathbb{N}}
\newcommand{\Z}{\mathbb{Z}}
\newcommand{\R}{\mathbb{R}}

\renewcommand{\H}{\mathbb{H}}

\newcommand{\SO}{\mathrm{SO}}
\newcommand{\K}{\mathcal{K}}

\newcommand{\wK}{\widetilde{\mathcal{K}}}

\newcommand{\wsigma}{\widetilde{\sigma}}
\newcommand{\wtau}{\widetilde{\tau}}
\newcommand{\vol}{\mathrm{Vol}}

\newcommand{\st}{\mathrm{st}}
\newcommand{\bH}{\partial_\infty\H}
\numberwithin{equation}{section}

\theoremstyle{plain}
\newtheorem{theorem}{Theorem}[section]

\newtheorem{lemma}[theorem]{Lemma}

\newtheorem*{claim*}{Claim}
\newtheorem*{theorem*}{Theorem}
\newtheorem*{proposition*}{Proposition}

\theoremstyle{definition}

\theoremstyle{remark}

\title{A note on the integrality of volumes of representations}
\author{Sungwoon Kim
}

\date{}
\begin{document}

\maketitle

\begin{abstract}
Let $\Gamma$ be a torsion-free, non-uniform lattice in $\SO(2n,1)$. We present an elementary, combinatorial--geometrical proof of a theorem of Bucher, Burger, and Iozzi in \cite{BBI21} which states that the volume of a representation $\rho:\Gamma\to\SO(2n,1)$, properly normalized, is an integer if $n$ is greater than or equal to  $2$.
\end{abstract}


\section{Introduction}

Let $\Gamma$ be a torsion-free lattice in $\SO(m,1)$ and $M=\Gamma\backslash\H^{m}$ be the quotient hyperbolic manifold. Let $\rho:\Gamma\to \SO(m,1)$ be a representation.
If $\Gamma$ is a uniform lattice, i.e., $M$ is compact, \emph{the volume of $\rho$} is defined by 
\begin{equation}\label{def:vol} \vol(\rho)=\int_{M}f^*\omega_{\H^{m}}.
\end{equation}
where $f:\H^{m}\to\H^{m}$ is a $\rho$-equivariant smooth map and $\omega_{\H^{m}}$ is the Riemannian volume form on $\H^{m}$.
The volume of $\rho$ is well defined, independent of the choice of $f$ and gives rise to a real-valued function on the representation variety $\mathrm{Hom}(\Gamma,\SO(m,1))$, $$\vol : \mathrm{Hom}(\Gamma,\SO(m,1))\to \R.$$

The function $\vol$ has played an important role in detecting discrete, faithful representations and studying the rigidity of hyperbolic lattices in the representation variety.
The function $\vol$ satisfies a Milnor--Wood type inequality:
\begin{equation}\label{eqn:MW} | \vol(\rho) | \leq \vol(i_\Gamma)=\vol(\Gamma\backslash\H^m) \end{equation}
with equality if and only if $\rho$ is discrete and faithful, where $i_\Gamma : \Gamma \to \SO(m,1)$ is the canonical inclusion.
A representation is said to be \emph{maximal} if equality holds in (\ref{eqn:MW}).
If $m\geq 3$, then, by Mostow's rigidity theorem, every discrete faithful representation of $\Gamma$ into $\SO(m,1)$ is conjugate to $i_\Gamma$ and hence every maximal representation is conjugate to $i_\Gamma$.
Another feature of the function $\vol$ is that it is constant on each connected component of $\mathrm{Hom}(\Gamma,\SO(m,1))$.
This implies that any continuous deformation of a maximal representation is still maximal.
Combining this with Mostow's rigidity theorem gives that if $m\geq 3$, any continuous deformation of $i_\Gamma$ is conjugate to $i_\Gamma$.
This is the so-called \emph{volume rigidity theorem}, which is stronger than the corresponding local rigidity theorem for uniform hyperbolic lattices.
Various approaches to the volumes of representations in the uniform lattice case have been provided by Thurston, Gromov, Goldman, Reznikov, and Besson--Courtois--Gallot.
For more details, we refer the reader to \cite{Th78,Gr81,Go82,Re96,BCG07}.

When $\Gamma$ is a torsion-free, non-uniform lattice in $\SO(m,1)$, one faces some problems in defining the volume of  a representation, due to the non-compactness of $M$.
One can try to define the volume of a representation in the same way as in the uniform lattice case above.
However, since $M$ is non-compact, the integral in (\ref{def:vol}) becomes an improper integral and thus the problem of integrability arises.
To overcome this problem, Dunfield used particular $\rho$-equivariant maps, called \emph{pseudo-developing maps}, and then partially proved the well-definedness of the volume of a representation \cite{Du99}.
Francaviglia later completely proved the well-definedness of this volume \cite{Fa04}.
Furthermore, as in the uniform lattice case, a Milnor--Wood type inequality and the volume rigidity theorem have been established in the non-uniform lattice case by \cite{Fa04,BBI13,KK16}.
Bucher, Burger, and Iozzi used bounded cohomology to define the volume of a representation and proved the volume rigidity theorem.
It was shown in \cite{KK16} that the definitions given by \cite{Fa04, BBI13} are equivalent.

The range of the function $\vol$ is an interesting part of the study of the volume of a representation.
In the uniform lattice case, the range of $\vol$ is finite, that is, the volume of a representation takes only finitely many values.
In the non-uniform lattice case, the function $\vol$ turns out to be continuous (See \cite{KK16,BBI21}).
Unlike the uniform lattice case, a non-empty interval is contained in the range of the function $\vol$ when $m=2,3$.
In particular, when $m=2$, the range of $\vol$ coincides with the interval $[\chi(M),-\chi(M)]$, where $\chi(M)$ is the Euler number of $M$.
For $m\geq 4$, Kim and Kim \cite{KK16} proved that the function $\vol$ in the non-uniform lattice case is also constant on each connected component, as in the uniform lattice case.
Furthermore, when $\Gamma$ is uniform and $m$ is even, it is well known that the volume of a representation $\rho:\Gamma\to\SO(m,1)$, properly normalized, is an integer.
In \cite{BBI21}, Bucher, Burger, and Iozzi extended this integrality to non-uniform lattices, as follows.

\begin{theorem}[Integrality theorem, \cite{BBI21}]\label{thm:BBI}
Let $\Gamma$ be a torsion-free, non-uniform lattice in $\SO(2n,1)$ and $\rho :\Gamma\to \SO(2n,1)$ be a representation.
Assume that $n\geq 2$.
Then the following holds.
\begin{itemize}
\item[(i)] If the manifold $M=\Gamma\backslash\H^{2n}$ has only toric cusps,
\[ \frac{2\vol(\rho)}{\vol(S^{2n})} \in \Z.\]
\item[(ii)] In general,
\[ \frac{2\vol(\rho)}{\vol(S^{2n})} \in \frac{1}{B_{2n-1}}\cdot\Z,\]
where $B_{2n-1}$ is the Bieberbach number in dimension $2n-1$.
\end{itemize}
\end{theorem}

Recall that the Bieberbach number is the smallest integer $B_d$  such that
any compact flat $d$-manifold has a covering of degree $B_d$ that is a torus.
Bucher, Burger, and Iozzi used the theory of bounded cohomology in order to verify the integrality of the volume of a representation.
The aim of the present paper is to give an elementary combinatorial--geometrical proof of the integrality theorem, which gives a geometric understanding of the integrality theorem.
The key ingredient of our apporach to the integrality theorem is that the volume of a geodesic simplex in $\H^{m}$ is computed by its generalized angle sum when $m$ is even.
More precisely, in \cite{Ho}, Hopf proved that 
\[
W(T)= 
\begin{cases} 
\ \displaystyle (-1)^{\frac{m}{2}}\frac{2\vol(T)}{\vol(S^m)} & \text{for $m$ even,} \\ 
\ 0 & \text{for $m$ odd}.
\end{cases}
\]
Roughly speaking, the generalized angle sum $W(T)$ of a geodesic simplex $T$ in $\H^{2n}$ is the alternating sum of all interior angles of $T$ assigned at the faces of $T$.

To give the ouline of our proof of the integrality theorem, let $\K$ be a triangulation of $M$ by geodesic simplices.
Then the volume of $M$ is the sum of the volumes of all geodesic simplices of $\K$, which is proportional to the sum of all generalized angle sums of all geodesic simplices of $\K$.
For each face $\tau$ of $\K$, the sum of the interior angles at $\tau$ of all geodesic simplices of $\K$ containing $\tau$ as a face equals $1$ if $M$ is a smooth manifold.
From this observation, Kellerhals and Zehrt \cite{KZ01} gave an elementary proof of the Gauss--Bonnet formula for a $2n$-dimensional complete hyperbolic manifold $M$ of finite volume 
\[ \vol(M)=(-1)^{n}\frac{\vol(S^{2n})}{2}\chi(M),\]
which is regarded as the volume of the canonical inclusion $i_\Gamma :\Gamma \to \SO(2n,1)$.

Similarly, the volume of an arbitrary representation $\rho:\Gamma\to \SO(2n,1)$ can be computed in terms of interior angles.
Let $f:\H^{2n}\to\H^{2n}$ be a $\rho$-equivariant pseudo-developing map such that $f$ sends each geodesic simplex of $\wK$ to a geodesic simplex in $\H^{2n}$, where $\wK$ is a triangulation of $\H^{2n}$ induced from the triangulation $\K$ of $M$.
Let $\tau$ be an arbitrary face of $\K$ and $\wtau$ be a lift of $\tau$ to $\H^{2n}$.
Denote by $W(\K,\tau;f)$ the sum of the interior angles at $f(\wtau)$ of all geodesic simplices $f(\wsigma)$ with $\wtau\subset \wsigma\in \wK$.
Note that $W(\K,\tau;f)$ does not depend on the choice of $\wtau$.
Then, one can verify the following equality:
\begin{align}\label{eqn:ve}
(-1)^{n}\frac{2\vol(\rho)}{\vol(S^{2n})}=\sum_{\tau\in\K} (-1)^{\dim \tau} W(\K,\tau;f).
\end{align}
To prove the integrality theorem, we first show that when $M$ has only toric cusps, every $W(\K,\tau;f)$ occurring in (\ref{eqn:ve}) is an integer, which immediately leads to the integrality theorem.
Then the integrality of $W(\K,\tau;f)$ in the toric cusp case enables us to prove the integrality theorem in the general case.


\section{Proof of the integrality theorem}
This section will be devoted to the combinatorial-geometrical proof of Theorem \ref{thm:BBI}.
Before proving the integrality theorem for non-uniform lattices, we give a proof of the integrality theorem for uniform lattices, which makes it easy to understand the main idea of our proof.


First of all, we begin by recalling from \cite{Ho} the notion of a generalized angle sum, which is the key ingredient of our approach.
Let $T$ be a geodesic simplex in $\H^{m}$ and $\tau$ be a face of $T$.
Choose an interior point $x$ of $\tau$ and an $(m-1)$-sphere $S(x,r)$ of radius $r>0$ at $x$.
Then the interior angle of $\sigma$ at the apex $\tau$ is defined by 
\[ W(T,\tau)=\frac{\vol_{m-1}(T\cap S(x,r))}{\vol_{m-1}(S(x,r))}.\]
Note that $W(T,\tau)$ is independent of the choice of $x$ and $r$.
If $\tau$ belongs to the ideal boundary $\partial_\infty \H^{m}$ of $\H^{m}$, the interior angle of $T$ at $\tau$ is defined to be zero.
The \emph{generalized angle sum} $W(T)$ of $T$ is defined by 
\[ W(T)=\sum_{\tau\in T} (-1)^{\dim \tau} W(T,\tau),\]
and satisfies
\[
W(T)= 
\begin{cases} 
\ \displaystyle (-1)^{\frac{m}{2}}\frac{2\vol(T)}{\vol(S^m)} & \text{for $m$ even,} \\ 
\ 0 & \text{for $m$ odd}.
\end{cases}
\]
Thus, in hyperbolic spaces of even dimension, the volume of a geodesic simplex can be computed by its generalized angle sum.
For more details, see \cite{Ho}.

Throughout this section, let $\Gamma$ denote a torsion-free lattice in $\SO(2n,1)$ and $M= \Gamma\backslash \H^{2n}$  be the quotient manifold.
Let $\K$ be a triangulation of $M$ and $\wK$ its lift to a triangulation of $\widetilde M=\H^{2n}$.

\subsection{Uniform lattice case}\label{sec:uniform}
We first give a combinatorial--geometrical proof of the following well-known integrality theorem for uniform lattices.

\begin{theorem}\label{thm:compact}
Let $\Gamma$ be a torsion-free uniform lattice in $\SO(2n,1)$ and $\rho :\Gamma\to \SO(2n,1)$ be a representation.
Assume that $n\geq 2$.
Then,
\[ \frac{2\vol(\rho)}{\vol(S^{2n})} \in \Z.\]
\end{theorem}
\begin{proof}
Let $\K$ be a triangulation of $M=\Gamma\backslash\H^{2n}$.
Denote by $\K_i$ the collection of all $i$-dimensional faces of $\K$.
The sum $\sum_{\sigma\in \K_{2n}}\sigma$ of all top-dimensional simplices in $\K$ yields a fundamental cycle of $M$.
Given a representation $\rho : \Gamma\to \SO(2n,1)$, the volume of $\rho$ is computed by 
\[ \vol(\rho)=\int_{M}f^*\omega_{\H^{2n}}=\sum_{\sigma\in \K_{2n}} \int_{\sigma}f^*\omega_{\H^{2n}}=\sum_{\sigma\in \K_{2n}} \epsilon(f,\sigma)\vol \left( \sigma;f  \right) \]
where $f :\H^{2n} \to \H^{2n}$ is a $\rho$-equivariant smooth map, $\widetilde \sigma$ is a lift of $\sigma$ to $\H^{2n}$, and \begin{displaymath}
\epsilon(f,\sigma)=\left\{ \begin{array}{ll} +1 & \textrm{if $f : \wsigma \rightarrow f(\wsigma)$ is orientation preserving,} \\
-1 & \textrm{otherwise} \end{array} \right.
\end{displaymath} and $\vol(\sigma;f)=\vol(f(\wsigma))$.
By the $\rho$-equivariance of $f$, both $\epsilon(f,\sigma)$ and $\vol(\sigma;f)$ are independent of the choice of $\wsigma$.
Furthermore, it turns out that $\vol(\rho)$ does not depend on either the choice of $f$ or the triangulation $\K$ of $M$.
Lemma 5.2 in \cite{BCG07} provides  a $\rho$-equivariant map $f:\H^{2n}\to\H^{2n}$ which is non-degenerate in a certain sense which will shortly be made precise.

\begin{lemma}[Besson--Courtois--Gallot]\label{lem:BCGmap}
Let $\Gamma$ be a uniform lattice in $\SO(m,1)$.
Given a representation $\rho : \Gamma \to \SO(m,1)$, there exist a triangulation $\K$ of $\Gamma\backslash\H^{m}$ and a piecewise continuous $\rho$-equivariant and affine map $f:\H^{m}\to\H^{m}$ which is non-degenerate in the sense that the image of each non-degenerate geodesic simplex of $\wK$ by $f$ is a non-degenerate geodesic simplex in $\H^{m}$.
\end{lemma}

We say that a $\rho$-equivariant map $f:\H^{m}\to\H^{m}$ is \emph{non-degenerate with respect to $\K$} if it is non-degenerate in the sense of Besson, Courtois, and Gallot as above.
From now on, we take both $\K$ and $f:\H^{2n}\to\H^{2n}$ as in Lemma \ref{lem:BCGmap}.
As seen before, the volume of each top-dimensional geodesic simplex $T$ in $\H^{2n}$ is computed by a generalized angle sum $W(T)$ as follows: 
\[ \vol(T)= \frac{(-1)^{n}}{2}\vol(S^{2n})W(T).\]

Let $\sigma$ be a top-dimensional simplex of $\K$ and $\tau$ be a face of $\sigma$.
Let $\wsigma$ be a lift of $\sigma$ to $\H^{2n}$ and $\wtau$ be a face of $\wsigma$ which is a lift of $\tau$.
Then we denote the interior angle of $f(\wsigma)$ at the apex $f(\wtau)$ by $W(\sigma,\tau;f)$ and the generalized angle sum of $f(\wsigma)$ by $W(\sigma;f)$, i.e., \[W(\sigma,\tau;f)=W(f(\wsigma),f(\wtau)) \text{ and }
 W(\sigma;f)=\sum_{\tau\subset \sigma} (-1)^{\dim \tau} W(\sigma,\tau;f).\]
It can be easily seen that both $W(\sigma,\tau;f)$ and $W(\sigma;f)$ are independent of the choice of $\wsigma$.
For convenience, we define another notation
\[ W(\K,\tau;f):=\sum_{\substack{\sigma \in \K_{2n} \\ \tau\subset \sigma}} \epsilon(f,\sigma)W(\sigma, \tau;f) \]
For instance, to compute the volume of the canonical inclusion $i_\Gamma:\Gamma\to\SO(2n,1)$, one can take $f=id_{\H^{2n}} :\H^{2n}\to\H^{2n}$ and then one can easily see that $W(\K,\tau;id_{\H^{2n}})=1$ for any $\tau\in \K$ and any triangulation $\K$ of $M$.

From the following equations, we can see that once the integrality of $W(\K,\tau;f)$ holds for all $\tau\in\K$, the integrality theorem for uniform lattices immediately follows.
\begin{align}
(-1)^{n}\frac{2\vol(\rho)}{\vol(S^{2n})}&=\sum_{\sigma\in \K_{2n}} \epsilon(f,\sigma)W(\sigma;f)  \nonumber \\
&= \sum_{\sigma\in \K_{2n}} \sum_{\tau\subset \sigma} (-1)^{\dim \tau}\epsilon(f,\sigma)W(\sigma, \tau;f)  \nonumber \\
&=\sum_{\tau\in\K} (-1)^{\dim \tau}\sum_{\substack{\sigma \in \K_{2n} \\ \tau\subset \sigma}} \epsilon(f,\sigma)W(\sigma, \tau;f)  \nonumber \\
&=\sum_{\tau\in\K} (-1)^{\dim \tau} W(\K,\tau;f) \label{eqn:integraltau}
\end{align}

The non-degenercy of $f$ in the sense of Lemma \ref{lem:BCGmap} allows us to prove that $W(\K,\tau;f)$ is always an integer, as follows.

\begin{lemma}\label{lem:integrality} Let $\rho : \Gamma \to \SO(2n,1)$ be a representation and $\K$ be a triangulation of $M=\Gamma\backslash \H^{2n}$.
If $f:\H^{2n}\to\H^{2n}$ is a piecewise continuous $\rho$-equivariant affine and non-degenerate map with respect to $\K$, then for any $\tau\in\K$,
\[ W(\K,\tau;f)=\sum_{\substack{\sigma \in \K_{2n} \\ \tau\subset \sigma}} \epsilon(f,\sigma)W(\sigma,\tau;f) \in \Z.
\]
where the sum runs over all $\sigma\in \K_{2n}$ containing $\tau$ as a face.
\end{lemma}

\begin{proof}
Let $\wtau$ be a lift of $\tau$ to $\H^{2n}$ and $x$ be an interior point of $\wtau$.
Since $\wK$ is a triangulation of $\H^{2n}$, the star $\st(\wtau)$ of $\wtau$, i.e., the union of simplices in $\wK$ having $\wtau$ as a face, is homeomorphic to a closed ball and thus its boundary $\partial \st(\wtau)$ is homeomorphic to the $(2n-1)$-sphere $S^{2n-1}$.
Define a map $f_\tau : \partial \st(\wtau) \to T_{f(x)}^1 \H^{2n}$ by assigning to each $y\in \partial \st(\wtau)$ the unit vector at $f(x) \in \H^{2n}$ corresponding to the geodesic ray starting from $f(x)$ towards $f(y)$.
This is possible due to the non-degeneracy of $f$.
Let $[\partial \st(\wtau))]_\Z$ and $[T_{f(x)}^1\H^{2n}]_\Z$ be the integral fundamental classes of $\partial \st(\wtau))$ and $T_{f(x)}^1\H^{2n}$ in degree $2n-1$ respectively.
Indeed, $H_{2n-1}(\partial \st(\wtau));\Z)\cong H_{2n-1}(T_{f(x)}^1\H^{2n};\Z)\cong H_{2n-1}(S^{2n-1};\Z)\cong \Z$.
Define the degree of $f$ to $\tau$ by 
\[ (f_\tau)_*[\partial \st(\wtau))]_\Z=(\deg_\tau f)\cdot [T_{f(x)}^1\H^{2n}]_\Z.
\]
Clearly $\deg_\tau f \in \Z$ and furthemore, one can easily deduce that 
\[ W(\K,\tau;f)=\sum_{\substack{\sigma \in \K_n \\ \tau\subset \sigma}} \epsilon(f,\sigma)W(\sigma,\tau;f) = \deg_\tau f \in \Z \]
which completes the proof.
\end{proof}

Finally, as mentioned above, Lemma \ref{lem:integrality} and the expression (\ref{eqn:integraltau}) immediately prove Theorem \ref{thm:compact}, that is, the integrality theorem for uniform lattices.
\end{proof}

\subsection{The case of a non-uniform lattice}\label{sec:nonuniform}

We now deal with the integrality theorem for non-uniform lattices.
That is, we will give a combinatorial-geometrical proof of Theorem \ref{thm:BBI}.
Suppose that $\Gamma$ is a torsion-free, non-uniform lattice in $\SO(m,1)$.
Then the quotient manifold $M= \Gamma\backslash\H^{m}$ is a complete, non-compact hyperbolic manifold of finite volume.
It is well known that $M$ admits a decomposition $$M=M_0\cup E_1 \cup \cdots \cup E_l$$ of $M$ into a compact manifold $M_0$ with boundary and disjoint non-compact ends $E_1,\ldots,E_l$ of finite volume.
Each end $E_i$ is diffeomorphic to $N_i\times (0,\infty)$ where $N_i$ is diffeomorphic to a compact Euclidean manifold.
More concretely, there exist disjoint horoballs $B_1,\ldots,B_l \subset \H^{m}$ based at $\xi_1,\ldots,\xi_l \in \bH^m$ respectively such that one can write $E_i=\Gamma_{B_i}\backslash B_i$, $N_i=\Gamma_{B_i}\backslash\partial B_i$ for $i=1,\ldots,l$ and $$M_0=\Gamma\backslash\left(\H^{m}\setminus \cup_{i=1}^l \Gamma\cdot B_i \right).$$
Note that the connected components of $\partial M_0$ consist of $N_1,\ldots,N_l$ and $\partial E_i=N_i$ for $i=1,\ldots,l$.
One can obtain $M$ by gluing $M_0$ and $E_1,\ldots,E_l$ along their respective boundaries $N_1,\ldots,N_l$.

One can compactify $M$ by adding one point to each end of $M$, the so-called \emph{end compactification of $M$}.
The end compactification of $M$ is realized by 
\[ \overline M=\Gamma\backslash\left(\H^{m}\cup \cup_{i=1}^l\Gamma\cdot \xi_i\right).
\]
Denote by $c_i$ the projection of $\Gamma\cdot \xi_i$ into $\overline M$ for $i=1,\ldots,l$.
Then $c_i$ is the one point added to the end $E_i$.
Each $c_i$ corresponds to a cusp point for $M$.

We will triangulate $\overline M$ in the following way: first, triangulate the compact manifold $M_0$ with boundary.
Denote by $\K_{M_0}$ the triangulation of $M_0$.
It induces a triangulation of the boundary $\partial M_0=N_1\cup\ldots\cup N_l$ of $M_0$.
Lifting the triangulation of $\partial M_0$ to $\H^{m}$, each horosphere $\partial B_i$ is dissected into simplices.
The collection of geodesic cones on simplices of $\partial B_i$ with the top point $\xi_i$ yields a $\Gamma_{B_i}$-invariant triangulation of $\overline{B}_i=B_i\cup \partial B_i \cup \{\xi_i\}$ and hence its projection to $\overline M$ induces a triangulation of $\overline E_i=E_i\cup\partial E_i\cup\{c_i\}$, denoted by $\K_{\overline E_i}$.
By adding to the simplices of $M_0$ all simplices of the triangulation $\K_{\overline E_i}$ of $\overline E_i$ for $i=1,\ldots,l$, one gets a triangulation $\K$ of $\overline M$.
Clearly $\K$ admits a decomposition $$ \K=\K_{M_0}\cup \K_{\overline{E}_1}\cup \cdots \cup \K_{\overline E_l}.$$
Note that each simplex of $\K$ is a simplex of either $M_0$ or $\overline E_1\cup\cdots\cup \overline E_l$ and, moreover, any lift of a simplex of $\K_{\overline E_i}$ to $\overline{B}_i$ is a geodesic cone on a simplex of $\partial B_i$ with the top point $\xi_i$.
Let $\K_\infty=\{c_1,\ldots,c_l\}$.
By straightening all simplices of $\K$, we can assume that every simplex of $\K$ is a geodesic simplex.
Let $\wK$ be the lifted triangulation of $\K$ to $\H^{m}\cup \cup_{i=1}^l \Gamma\cdot \xi_i$.
Let $\widetilde E_i$ be the inverse image of $E_i$ by the covering map $\H^{m}\to \Gamma\backslash\H^{m}=M$ for each $i=1,\ldots,l$.
Such a triangulation $\K$ of $M$ is called a \emph{generalized triangulation} of $M$.
Throughout this section, we assume that $\K$ is a generalized triangulation of $M$.

Given a representation $\rho:\Gamma\to \SO(m,1)$, one can construct a piecewise continuous $\rho$-equivariant and affine function $f:\H^{m}\to\H^{m}$ by the same method as that given by Besson--Courtois--Gallot in \cite{BCG07} as follows: Let $D\subset \H^{m}$ be a Dirichlet domain for $\Gamma$.
By perturbing $D$ a little, one can assume that no vertex of the triangulation $\wK$ in $\H^{m}$ is on the boundary of $D$.
Let $\{x_1,\ldots,x_N\}$ be the set of vertices of $\wK$ in the interior of $D$.
Choose points $y_1,\ldots,y_N\in \H^{m}$ and $\eta_1,\ldots,\eta_l \in \H^{m}\cup\partial_\infty \H^{m}$ so that each $\eta_i$ is a fixed point of $\rho(\Gamma_i)$ for $i=1,\ldots,l$.
Then set $f(x_i)=y_i$, $f(\xi_i)=\eta_i$, $f(\gamma \cdot x_i)=\rho(\gamma)\cdot y_i$ and $f(\gamma \cdot \xi_i)=\rho(\gamma)\cdot \eta_i$ for all $i=1,\ldots,N$ and all $\gamma\in \Gamma$.
Note that the images under $f$ of the vertices of $\wK$ are chosen $\rho$-equivariantly.
On each geodesic simplex $[v_0,\ldots,v_m]$ of $\wK$, define an affine map from $[v_0,\ldots,v_m]$ to $[f(v_0),\ldots,f(v_m)]$.
Gluing all such affine maps over all geodesic simplices of $\wK$ yields a piecewise continuous $\rho$-equivariant and affine map $f:\H^{m}\to\H^{m}$.
According to the proof of \cite[Lemma 5.2]{BCG07}, it is possible to choose $y_1,\ldots,y_N$ and $\eta_1,\ldots,\eta_l$ so that $f$ is non-degenerate in the sense of Lemma \ref{lem:BCGmap}.
For a detailed proof, we refer to the proof of \cite[Lemma 5.2]{BCG07} or \cite[Lemma 5.2]{KK16}.

In the proof of the integrality theorem for uniform lattices in Section \ref{sec:uniform}, Lemma \ref{lem:integrality}, which proves the integrality of $W(\K,\tau;f)$ for all $\tau\in\K$, is essential.
In the case of a non-uniform lattice, the integrality of $W(\K,\tau;f)$ does not always hold, unlike for uniform lattices.
However, it turns out that the integrality of $W(\K,\tau;f)$ still holds in the case of a non-uniform lattice for all $\tau\in\K$ except cusp points, which can be seen as follows.

\begin{lemma}\label{lem:ins}
Let $\Gamma$ be a torsion-free non-uniform lattice in $\SO(2n,1)$ and $\rho:\Gamma\to\SO(2n,1)$ be a representation. Let 
$f:\H^{2n}\to\H^{2n}$ be a piecewise continuous $\rho$-equivariant affine and non-degenerate map with respect to a generalized triangulation $\K$ of $M$.
Then, for all $\tau \in \K\setminus \K_\infty$,
$$W(\K,\tau;f)\in \Z.$$
\end{lemma}
\begin{proof}
If $\tau \in \K\setminus \K_\infty$, any interior point $x$ of $\tau$ is placed inside $\H^{2n}$, which makes it possible to define the map $f_\tau : \partial \st(\wtau)\to T^1_{f(x)}\H^{2n}$ as in the proof of Lemma \ref{lem:integrality}.
Then the exactly same proof of Lemma \ref{lem:integrality} works as well.
Therefore we conclude that $W(\K,\tau;f)=\deg f_\tau \in \Z$.
\end{proof}

It is worth pointing out that due to Lemma \ref{lem:ins}, it is sufficient to prove $W(\K,c_i;f)\in \Z$ for all $c_i \in\K_\infty$ in order to show Theorem \ref{thm:BBI}(i).

\subsubsection{The case of toric cusps}\label{sec:toric}

We first deal with the case where $M$ has only toric cusps.
We will construct a specific piecewise continuous $\rho$-equivariant affine map $f_0:\H^{2n}\to\H^{2n}$ by choosing $y_1,\ldots,y_N$ and $\eta_1,\ldots,\eta_l$ properly so that $W(\sigma,c_i;f_0)=0$ for all $\sigma\in\K$ having $c_i$ as a vertex and all $i=1,\ldots,l$.
First, note that each $\rho(\Gamma_{B_i})$ has a fixed point in either $\partial_\infty \H^{2n}$ or $\H^{2n}$.
If $\rho(\Gamma_{B_i})$ has a fixed point in $\partial_\infty\H^{2n}$, choose $\eta_i^0 \in \partial_\infty\H^{2n}$ and set $f_0(\xi_i)=\eta_i^0$.
Then it immediately follows from the fact that $\eta_i^0 \in \partial_\infty\H^{2n}$ that $W(\sigma,c_i;f_0)=W(f_0(\wsigma),f_0(\xi_i))=W(f_0(\wsigma),\eta_i^0)=0$ for any simplex $\wsigma\in\wK$ having $\xi_i$ as a vertex.
Otherwise (i.e., if $\rho(\Gamma_{B_i})$ has only fixed points in $\H^{2n}$), by the assumption that every end of $M$ is a toric cusp, $\rho(\Gamma_{B_i})$ is conjugate to a subgroup of $\mathrm{S}(\mathrm{O}(2)^{n} \times \mathrm{O}(1))$ (see \cite[Lemma 5.2]{BBI21}).
Hence,
$\rho(\Gamma_{B_i})$ stabilizes a proper totally geodesic subspace $H^2_i$ which is isometric to $\H^2$ and, moreover, fixes a point $\eta_i^0$ in $H^2_i$.
We set $f_0(\xi_i)=\eta_i^0$ in this case.
Choose $y_1^0,\ldots, y_N^0 \in\H^{2n}$ to define the function $f_0:\H^{2n}\to\H^{2n}$ so that every vertex of $\wK$ on $\partial B_i$ is mapped into $H^2_i$ by $f_0$ and, moreover, we require that for each $\sigma\in \K$, the restriction map of $f_0$ to the vertices of $\wsigma$ is injective.
If $n\geq 2$, for every $\sigma\in\K_{\overline E_i}$, $f_0(\wsigma)$ is degenerate since all vertices of $\wsigma$ are contained in a $2$-dimensional totally geodesic subspace of $\H^{2n}$ which is isometric to $\H^2$.
Summarizing, $f_0$ maps each geodesic simplex of the ends of $M$ to either a geodesic simplex with ideal top point or a degenerate geodesic simplex that is contained in a copy of $\H^2$.
We fix such a map $f_0$.

Next we will define a piecewise continuous $\rho$-equivariant affine and non-degenerate map $f:\H^{2n}\to\H^{2n}$ which is sufficiently close to $f_0$.
This $f$ will play an important role in our proof of the integrality theorem.
We say that a map $g:\H^{2n}\to\H^{2n}$ is \emph{$r$-close} to $f_0$ with respect to $\K$ if $d(g(v),f_0(v))<r$ for all vertices $v$ of $\wK$.
Following the proof of \cite[Lemma 5.2]{BCG07}, for each $r>0$ one can choose $y_i^r\in B(y_i,r)$ for $i=1,\ldots,N$ so that the function $f$ defined by using $y_1^r, \ldots, y_N^r$ instead of $y_1^0,\ldots,y_N^0$ is non-degenerate.
Then $f$ is a piecewise continuous $\rho$-equivariant affine and non-degenerate map that is $r$-close to $f_0$ with respect to $\K$.

\begin{lemma}\label{lem:const}
There is a constant $\delta>0$ such that if $f$ and $g$ are piecewise continuous $\rho$-equivariant affine and non-degenerate maps that are $\delta$-close to $f_0$ with respect to $\K$, then for all $\tau\in \K$,
$$W(\K,\tau;f)=W(\K,\tau,g),$$
and furthermore $W(\K,c_i;f)=W(\K,c_i;g)=0$ for all $c_i\in\K_\infty$.
\end{lemma}
\begin{proof}
First, recall that the restriction of $f_0$ to the vertex set of $\wsigma$ is injective for each $\sigma\in\K$.
This means that there is a positive number $d_\sigma>0$ such that any two vertices of $f_0(\wsigma)$ are apart by a distance of at least $d_\sigma$.
Let $d_0=\min\{d_\sigma \mid \sigma\in\K\}/4$.
Then any two vertices of $f_0(\wsigma)$ are apart by a distance of at least $4d_0$ for every $\sigma \in\K$.
Note that $d_0$ is positive since $\K$ is a finite set.
Roughly speaking, any two vertices of $f_0(\wsigma)$ are far enough apart.

Let $f$ and $g$ be piecewise continuous $\rho$-equivariant affine and non-degenerate maps that are $r$-close to $f_0$ with respect to $\K$.
For any $v\in\wK$, 
\begin{align}\label{eqn:nbd} d(f_0(v),f(v))<r \text{ and } d(f_0(v), g(v))<r.\end{align}
Assume that $r<d_0$.
From (\ref{eqn:nbd}) and the fact that
any two vertices of $f_0(\wsigma)$ are apart by a distance of at least $4d_0$, it can be easily seen that any two vertices of $f(\wsigma)$ or $g(\wsigma)$ are apart by a distance of at least $4d_0-2r>2d_0$.
The corresponding vertices of $f(\wsigma)$ and $g(\wsigma)$ are apart by a distance of at most $2r$.
In this situation, one can deduce that as $r$ gets closer to $0$, the corresponding vertices of $f(\wsigma)$ and $g(\wsigma)$ get closer to each other and, furthermore, their corresponding angles also get closer to each other.
Therefore there is a constant $\delta>0$ such that for all piecewise continuous $\rho$-equivariant affine and non-degenerate maps $f, g$ that are $\delta$-close to $f_0$ with respect to $\K$ and all $\tau\in \K$,
\begin{align}\label{eqn:res} | W(\K,\tau;f) -W(\K,\tau;g) | <1/2 \end{align}
By Lemma \ref{lem:ins}, both $W(\K,\tau;f)$ and $W(\K,\tau;f)$ are integers for all $\tau \in \K\setminus \K_\infty$.
Combining the integrality of $W(\K,\tau;f)$ and $W(\K,\tau;g)$ for $\tau \in \K\setminus \K_\infty$ with (\ref{eqn:res}), it is easy to see that for all $\tau\in\K\setminus\K_\infty$, 
$$W(\K,\tau;f)=W(\K,\tau,g).$$
Now it remains to prove the lemma for $c_i\in\K_\infty$.
Indeed, it is sufficient to prove that $W(\K,c_i;f)=0$ for all $c_i\in\K_\infty$.

Let $f$ be a piecewise continuous $\rho$-equivariant affine and non-degenerate map that is $\delta$-close to $f_0$ with respect to $\K$.
Fix $c_i\in\K_\infty$ where $c_i$ is the cusp point corresponding to the end $E_i$ of $M$.
Then for each $k\in\N$, define a piecewise continuous $\rho$-equivariant affine and non-degenerate map $f_k : \H^{2n}\to \H^{2n}$ by setting $f_k(v)=f(v)$ if $v$ is not a vertex in $\widetilde{E}_i$ and $f_k(v) \in B_{\delta/k}(f_0(v))$ $\rho$-equivariantly if $v$ is a vertex in $\widetilde{E}_i$.
Obviously every $f_k$ is $\delta$-close to $f_0$ with respect to $\K$.
Then
\begin{align}
(-1)^{n}\frac{2\vol(\rho)}{\vol(S^{2n})}&=\sum_{\tau\in\K\setminus\K_\infty} (-1)^{\dim \tau} W(\K,\tau;f_k)+\sum_{\tau\in\K_\infty} (-1)^{\dim \tau} W(\K,\tau;f_k) \label{eqn:exp1}\\
&=\sum_{\tau\in\K\setminus\K_\infty} (-1)^{\dim \tau} W(\K,\tau;f)+\sum_{\tau\in\K_\infty} (-1)^{\dim \tau} W(\K,\tau;f) \label{eqn:exp2}
\end{align}
Since $f_k$ is $\delta$-close to $f_0$ with respect to $\K$, the above proof implies that $W(\K,\tau,f_k)=W(\K,\tau,f)$ for all $\tau\in \K\setminus\K_\infty$ and all $k\in \N$.
Furthermore, by the definition of $f_k$, $W(\K,c_j,f_k)=W(\K,c_j,f)$ for all $j\neq i$ and all $k\in\N$.
Thus from (\ref{eqn:exp1}) and (\ref{eqn:exp2}), we have that for all $k\in\N$,
$$W(\K,c_i,f_k)=W(\K,c_i,f).$$
As $k$ goes to infinity, $f_k$ converges to $f_0$ on $\widetilde{E}_i$ and thus the opposite $(2n-1)$-face to $f_k(\xi_i)$ in $f_k(\wsigma)$ gets closer to a degenerate geodesic simplex contained in a copy of $\H^2$, which means that
$W(\sigma,c_i;f_k)$ goes to $0$.
Therefore,
 $$W(\K,c_i;f)=\lim_{k\to\infty} W(\K,c_i,f)=\lim_{k\to\infty} W(\K,c_i,f_k)=0$$
which completes the proof.
\end{proof}

We are ready to prove Theorem \ref{thm:BBI}(i).
In fact, Lemma \ref{lem:const} is essential in the proof of Theorem \ref{thm:BBI}(i).

\begin{proof}[Proof of Theorem \ref{thm:BBI}(i)]
Lemmas \ref{lem:ins} and \ref{lem:const} imply that for any piecewise continuous $\rho$-equivariant affine and non-degenerate map $f$ that is $\delta$-close to $f_0$ with respect to $\K$, $$W(\K,\tau;f)\in \Z \text{ for all }\tau \in \K.$$
Then from the equation
\begin{align*}
(-1)^{n}\frac{2\vol(\rho)}{\vol(S^{2n})}
&=\sum_{\tau\in\K} (-1)^{\dim \tau} W(\K,\tau;f),
\end{align*} 
Theorem \ref{thm:BBI}(i) immediately follows.
\end{proof}

We have shown that for every piecewise continuous $\rho$-equivariant affine and non-degenerate map $f :\H^{2n}\to\H^{2n}$ that is sufficiently close to a specific $\rho$-equivariant degenerate map $f_0 :\H^{2n}\to\H^{2n}$, all $W(\K,\tau;f)$ are integers.
Indeed, this property of integrality holds for any piecewise continuous $\rho$-equivariant affine and non-degenerate map, as follows.

\begin{lemma}\label{lem:toric}
Let $f :\H^{2n} \to \H^{2n}$ be an arbitrary piecewise continuous $\rho$-equivariant affine and non-degenerate map with respect to $\K$.
Then for all $\tau \in \K$,
\[ W(\K,\tau;f)\in \Z.\]
\end{lemma}
\begin{proof}
By Lemma \ref{lem:ins}, $W(\K,\tau;f)\in \Z$ for all $\tau\in \K\setminus\K_\infty$.
Thus it is sufficient to prove that $ W(\K,c_i;f)\in \Z$ for all $c_i \in \K_\infty$.
Since the ends $E_1,\ldots,E_l$ of $M$ are pairwise disjoint, it is possible to choose the images of vertices of $\widetilde E_1,\ldots,\widetilde E_l$ independently.
Hence for each $i=1,\ldots,l$, it is possible to construct a piecewise continuous $\rho$-equivariant affine and non-degenerate map $f_i$ such that $f_i$ coincides with $f$ on the vertex set of $\widetilde E_i$ and $f_i$ is $\delta$-close to $f_0$ with respect to the vertex sets of the other ends $\widetilde E_1,\ldots,\widetilde E_{i-1},\widetilde E_{i+1},\ldots,\widetilde E_l$.
By Lemma \ref{lem:const}, $W(\K,c_j;f_i)=0$ for all $j\neq i$ and $W(\K,c_i;f_i)=W(\K,c_i;f)$.
Then we have that
\begin{align*}
(-1)^{n}\frac{2\vol(\rho)}{\vol(S^{2n})}
&=\sum_{\tau\in\K} (-1)^{\dim \tau} W(\K,\tau;f_i)\\
&=\sum_{\tau\in\K\setminus \K_\infty} (-1)^{\dim \tau}W(\K,\tau;f_i)+\sum_{j=1}^l W(\K,c_j;f_i)  \\
&=\sum_{\tau\in\K\setminus \K_\infty} (-1)^{\dim \tau}W(\K,\tau;f_i)+W(\K,c_i;f).
\end{align*}
According to Theorem \ref{thm:BBI}(i) and Lemma \ref{lem:ins}, in the above equation, all terms except for $W(\K,c_i;f)$ are integers, which leads us to conclude that $W(\K,c_i;f)\in\Z$.
\end{proof}

We will use Lemma \ref{lem:toric} to prove the integrality theorem in the general case.

\subsubsection{The general case}

Finally, we will prove the integrality theorem for the general case.
Let $\Gamma$ be a general torsion-free non-uniform lattice in $\SO(2n,1)$ and $M=\Gamma\backslash \H^{2n}$.
Then $\Gamma$ has a subgroup $\Gamma'$ of finite index such that the quotient manifold $M'=\Gamma'\backslash\H^{2n}$ has only toric cusps.
Let $E_1,\ldots, E_l$ be the disjoint non-compact ends of $M$.
Recall that there exist disjoint horoballs $B_1,\ldots,B_l \subset \H^{2n}$ based at $\xi_1,\ldots,\xi_l \in \bH^n$ respectively such that one can write $E_i=\Gamma_{B_i}\backslash B_i$ and $E_i$ is homeomorphic to $N_i\times (0,\infty)$ where $N_i=\Gamma_{B_i}\backslash\partial B_i$ for $i=1,\ldots,l$.
Let $E_i'=\Gamma_{B_i}'\backslash B_i$ and $N_i'=\Gamma_{B_i}'\backslash\partial B_i$ for $i=1,\ldots,l$.
Then these $E_1',\ldots,E_l'$ are the pairwsie disjoint non-compact ends of $M'$.
Furthermore the covering map $\pi :M'\to M$ induces covering maps $E_i'\to E_i$ and $N_i'\to N_i$.
Note that each $N_i'$ is a flat torus by assumption.

Let $\overline M$ and $\overline{M'}$ be the end compactifications of $M$ and $M'$ respectively.
Denote by $c_i$ and $c_i'$ the end points added to $E_i$ and $E_i'$ respectively for $i=1,\ldots,l$.
Then the covering map $\pi :M'\to M$ is extended to the map ${\pi}:\overline{M'}\to \overline{M}$ by setting
\[{\pi}(c_i')=c_i \text{ for } i=1,\ldots,l.\]

Choose a triangulation $\K$ of $\overline M$ as constructed in the beginning of Section \ref{sec:nonuniform} and get a triangulation $\K'$ of $\overline{M'}$ by lifting the triangulation $\K$ to $\overline{M}$ by $\pi$.
It is obvious that both $\K$ and $\K'$ induce the same lifted triangulation of $\H^{2n}\cup \cup_{i=1}^l\Gamma\cdot \xi_i$.

Given a representation $\rho :\Gamma\to\SO(2n,1)$, let $\rho':\Gamma'\to\SO(2n,1)$ be the restriction of $\rho$ to $\Gamma'\subset \Gamma$.
Choose a piecewise continuous, $\rho$-equivariant, affine and non-degenerate map $f:\H^{2n}\to \H^{2n}$ with respect to $\K$.
Clearly the map $f$ is also a piecewise continuous, $\rho'$-equivariant, affine and non-degenerate map with respect to $\K'$.
Let $\tau\in \K$ and $x$ be an interior point of $\tau$.
Let $\tau'$ and $x'$ be the lifts of $\tau$ and $x$ to $M'$ respectively so that $x'$ is an interior point of $\tau'$.
If $\tau\notin \K_\infty=\{c_1,\ldots,c_l\}$, then $x$ is in $M$.
Thus there exists a neighborhood $U'$ of $x'$ such that $U'$ is mapped isometrically onto a neighborhood $U:=\pi(U')$ of $x$ by the covering map $\pi : M' \to M$.
Noting that $\K'$ is the lifted triangulation of $\K$ to $M'$ by $\pi$, it is easy to deduce that
\begin{equation}\label{eqn:int} W(K,\tau;f)=W(K',\tau';f).\end{equation}
Since $M'$ has only toric cusps, Lemma \ref{lem:toric} implies $W(K',\tau';f)\in \Z$ and thus (\ref{eqn:int}) implies $W(K,\tau;f)\in \Z$ for all $\tau\in \K\setminus\K_\infty$.

If $\tau\in \K_\infty$, the situation is a little different.
Let $\tau=c_i$ for some $i=1,\ldots,l$.
Recall that every simplex $\sigma\in E_i$ containing $c_i$ as a vertex is a geodesic cone on a simplex of $N_i=\partial E_i$ with the top point $c_i$.
Therefore, it can be easily seen that for a sufficiently fine triangulation $\K$, the number of lifts of $\sigma$ by the covering map $E_i'\to E_i$ equals the degree of the covering map $N_i'\to N_i$.
Moreover, each lift of $\sigma$ is a geodesic simplex on a simplex of $N_i'=\partial E_i'$ with the top point $c_i'$.
Thereofore, we deduce that
\begin{equation}\label{eqn:end} \deg(N_i'\to N_i) \cdot W(\K,c_i;f)=W(\K',c_i';f).\end{equation}

Lemma \ref{lem:toric} implies $W(\K',c_i';f) \in \Z$.
Furthermore $\deg(N_i'\to N_i)$ divides $B_{2n-1}$ since $N_i'$ is a flat torus of dimension $2n-1$.
Thus for all $i=1,\ldots,l$, $$W(\K,c_i;f) =\frac{W(\K',c_i',f)}{\deg(N_i'\to N_i)} \in  \frac{W(\K',c_i',f)}{B_{2n-1}}\cdot \Z \subset \frac{1}{B_{2n-1}}\cdot \Z.$$
Summarizing so far, $W(\K,\tau;f)\in \Z$ if $\tau\notin \K_\infty$ and $W(\K,\tau;f)\in \frac{1}{B_{2n-1}}\cdot \Z$ if $\tau\in \K_\infty$.
Therefore,
\[ (-1)^{n}\frac{2\vol(\rho)}{\vol(S^{2n})} =\sum_{\tau\in\K} (-1)^{\dim \tau} W(\K,\tau;f) \in \frac{1}{B_{2n-1}}\cdot\Z,\]
which completes the proof of Theorem \ref{thm:BBI}(ii).~\qed

\vspace{1\baselineskip}\noindent
Department of Mathematics,
Jeju National University,
Jeju 63243,
Republic of Korea\\
\url{sungwoon@jejunu.ac.kr}

\end{document}